\documentclass[a4paper, 12pt]{amsart}
\usepackage{amsfonts, amsthm, amssymb, amsmath}
\usepackage{mathrsfs,array}
\usepackage{xy}
\usepackage{hyperref}
\usepackage{verbatim}
\usepackage[latin1]{inputenc}
\input xy
\xyoption{all}

\newtheorem{theo}{Theorem}[section]
\newtheorem{prop}[theo]{Proposition}
\newtheorem{defi}[theo]{Definition}
\newtheorem{lemm}[theo]{Lemma}
\newtheorem{coro}[theo]{Corollary}

\newcommand{\mb}{\mathbb}
\newcommand{\mc}{\mathcal}
\newcommand{\mf}{\mathfrak}

\newcommand{\ra}{\rightarrow}

\DeclareMathOperator{\Spm}{Spm}
\DeclareMathOperator{\MF}{MF}
\DeclareMathOperator{\Ord}{Ord}
\DeclareMathOperator{\soc}{soc}

\DeclareMathOperator{\Hom}{Hom}

\DeclareMathOperator{\End}{End}

\DeclareMathOperator{\Spec}{Spec}

\DeclareMathOperator{\Gal}{Gal}

\DeclareMathOperator{\Frob}{Frob}
\DeclareMathOperator{\Ind}{Ind}

\DeclareMathOperator{\GL}{GL}

\DeclareMathOperator{\Nm}{Nm}

\DeclareMathOperator{\HT}{HT}

\DeclareMathOperator{\Fil}{Fil}

\DeclareMathOperator{\ord}{ord}

\title[Weak local-global compatibility]{Weak local-global compatibility and ordinary representations}
\date {\today}
\author{Przemys\l aw Chojecki}
\email{chojecki@math.jussieu.fr}

\begin{document}

\begin{abstract}
We introduce a general formalism with minimal requirements under which we are able to prove the pro-modular Fontaine-Mazur conjecture. We verify it in the ordinary case using the recent construction of Breuil and Herzig. 
\end{abstract}
\maketitle

\tableofcontents

\section{Introduction}

\noindent In \cite{em2}, Emerton has shown that the completed cohomology of modular curves realises the $p$-adic local Langlands correspondence and used this result to prove the Fontaine-Mazur conjecture for $\GL_2(\mb{Q})$. We start from the observation that Emerton's methods can be well formalized to work for other groups, at least if we assume certain hypotheses, for example the existence of the $p$-adic Langlands correspondence. Fortunately, only a part of properties of the conjectural $p$-adic local Langlands correspondence are needed for applications to the pro-modular Fontaine-Mazur conjecture. We list them under hypothesis (H1) in the body of this text. After introducing this local definition, we move to the global setting. We work on the unitary Shimura varieties of type $U(n)$. After establishing certain basic results on the completed cohomology of these objects, we introduce the notion of an allowable set, which is a dense set of points on the eigenvariety, such that the specialisation at its points of a certain universal deformation of $\bar{\rho}$ lies in the completed cohomology of our Shimura varieties. This gives a necessary global condition to link the local hypothesis (H1) with the completed cohomology. Having to deal only with allowable sets is easier, as we may hope that the description of the $p$-adic Langlands correspondence for certain representations (regular and crystalline) will be explicit.

\medskip

\noindent We remark that eventually we use two deformation arguments: one at the local level and the other at the global level (the existence of allowable points). They are related to two hypotheses ((H1) and (H2) respectively) on our global Galois representation $\bar{\rho}$. Assuming also a mild hypothesis (H3), we are able to prove the pro-modular Fontaine-Mazur conjecture for $U(n)$ in the following form (actually, we develop even more general formalism):

\begin{prop} Let $F$ be a CM field and let $E$ be a finite extension of $\mb{Q}_p$. Let $\rho : \Gal (\bar{F}/F) \ra \GL_n(E)$ be a continuous Galois representation such that
\medskip

\noindent (1) $\rho$ is pro-modular.
\medskip

\noindent (2) $\rho _v$ is de Rham and regular for every $v | p$.

\medskip

\noindent (3) $\bar{\rho}$ satisfies hypotheses (H1)-(H3).

\medskip

\noindent Then $\rho$ is a twist of a Galois representation associated to an automorphic form on $U(n)$.
\end{prop}

\medskip

\noindent The pro-modularity condition is explained in Section 3. It should not be very restrictive, as it is believed that any representation $\rho$ for which $\bar{\rho}$ is modular, is pro-modular (this is proved by Emerton for $\GL_2$ over $\mb{Q}$).

\medskip

\noindent As a corollary to this proposition, we obtain a version of the Fontaine-Mazur conjecture on the respective eigenvariety.

\begin{coro}
Let $\bar{\rho} : \Gal (\bar{F}/F) \ra \GL_n(k)$ be a continuous Galois representation which satisfies hypotheses (H1)-(H3). Let $\mc{X}[\bar{\rho}]$ be the $\bar{\rho}$-part of the eigenvariety $\mc{X}$ associated to $U(n)$ by the construction of Emerton from \cite{em1}. Let $x \in \mc{X}[\bar{\rho}]$ be an $E$-point such that its associated representation $\rho _x : \Gal(\bar{F}/F) \ra \GL_n(E)$ is de Rham and regular at every place of $F$ above $p$. Then $x$ is modular.
\end{coro}

\medskip

There is one principal example (besides $\GL_2(\mb{Q}_p)$) when our formalism is satisfied and it was the motivation behind writing this text - namely, the recent construction of the ordinary representations of Breuil-Herzig (\cite{bh}). We review this setting in the second part of this paper and then we prove unconditionally the pro-modular Fontaine-Mazur conjecture for $U(n)$ in the ordinary totally indecomposable setting at the end of this text (see Theorem 4.18 and its Corollary 4.19). Interestingly, the proof is relatively simple and we do not use in it the full construction of Breuil-Herzig. Our main unconditional result is
\begin{theo}
Let $z \in X_{K^p}(E)$, where $X_{K^p}$ is the eigenvariety of some tame level $K^p$ associated to $U(n)$ and let $\rho$ be the Galois representation associated to $z$. For each $v \mid p$ we assume that
\begin{enumerate}
\item $\rho_v$ is ordinary, de Rham and regular;
\item the reduction $\bar{\rho} _{v}$ is generic and totally indecomposable.
\end{enumerate}
Then $z$ is classical (i.e. it arises from a classical automorphic representation of $U(n)$).
\end{theo}

\textbf{Acknowledgements.} We thank Jean-Francois Dat, Christian Johansson and Claus Sorensen for useful remarks regarding this text. We also thank John Bergdall, Christophe Breuil and Florian Herzig for a useful correspondence.

\section{Definitions and basic facts}



\noindent Let $L$ denote an imaginary quadratic field in which $p$ splits and let $c$ be the complex conjugation. Choose a prime $u$ above $p$. Let us denote by $F^+$ a totally real field of degree $d$. Set $F = L F^+$. We will assume that $p$ totally decomposes in $F$. Let $D/F$ be a central simple algebra of dimension $n^2$ such that $F$ is the centre of $D$, the opposite algebra $D ^{op}$ is isomorphic to $D \otimes _{L,c} L$ and $D$ is split at all primes above $u$. Choose an involution of the second kind $*$ on $D$ and assume that there exists a homomorphism $h : \mb{C} \ra D _{\mb{R}}$ for which $b \mapsto h(i)^{-1} b^* h(i)$ is a positive involution on $D _{\mb{R}}$.

\medskip

\noindent Define the reductive group
$$G(R) = \{ (\lambda ,g) \in R^{\times} \times  D^{op} \otimes _{\mb{Q}} R \ | \ g g^* = \lambda \}.$$
We assume that $G$ is a unitary group of signature $(0,n)$ at all infinite places.

\medskip

\noindent We choose a $p$-adic field $E$ with ring of integers $\mc{O}$ and residue field $k$. These will be our coefficient rings.

\medskip

\noindent We will fix an integral model of $G$ over $\mc{O} _{F^+}[1/N]$ (see for example 4.1 in \cite{bh} for details). We consider 0-dimensional Shimura varieties $S_K = G(\mb{Q}) \backslash G(\mb{A}_f) / K$ for $G$, where $K$ is a compact open subgroup of $G(\mb{A}_f)$. 

\medskip

\noindent Let $W$ be a finite-dimensional representation of $G$ over $E$. By the construction described in Chapter 2 of \cite{em1}, we can associate to $W$ a local system $\mathcal{V}_W$ on $S_K$.

\medskip

\noindent Let us fix a finite set $\Sigma$ of primes $w$ of $F$, such that $w _{| F^+}$ splits and $w$ does not divide $pN$. We can now define the abstract Hecke algebra 
$$\mb{T} _{\Sigma} ^{abs} = \mc{O}[T _{w} ^{(i)}] _{w \notin \Sigma}$$
where $T _{w} ^{(i)}$ are the Hecke operators for $1 \leq i \leq n$. The operator $T _{w} ^{(i)}$ acts on the Shimura variety $S_K$ by a double coset $\GL_n (\mc{O} _{F_w}) \left( \begin{smallmatrix} 1 _{n-j} & 0 \\ 0 & \varpi _w 1_j \end{smallmatrix} \right) \GL_n (\mc{O} _{F_w})$, where $\varpi _w$ is a uniformiser of $\mc{O} _{F_w}$.

\medskip

\noindent We define the completed cohomology of Emerton by
$$\widehat{H}^0(K^p) = (\varprojlim _s \varinjlim _{K_p} H^0(S_{K_pK^p} ,\mc{O} / \varpi ^s \mc{O} )) \otimes _{\mc{O}} E$$
where $K_p \subset G(\mb{Q}_p)$ and $K^p \subset G(\mb{A} _f ^p)$ are open compact subgroups. We also define its $\mc{O}$-submodule
$$\widehat{H}^0(K^p)_{\mc{O}} = \varprojlim _s \varinjlim _{K_p} H^0(S_{K_pK^p} ,\mc{O} / \varpi ^s \mc{O} )$$
We will fix the tame level $K^p$ for the rest of the text. Let $K^{\Sigma }= \prod _{l \notin \Sigma} G(\mb{Z} _l)$. We assume that $(K^p)_l = G(\mb{Z}_l)$ at each $l\notin \Sigma$.

\medskip

\noindent We write $\mb{T}(K_pK^p)$ for the image of $\mb{T}_{\Sigma} ^{abs}$ in $\End _{\mc{O}}(H^0(S_{K_pK^p}, \mc{O}))$. Then we define 
$$\mb{T} = \mb{T}(K^p)  := \varprojlim _{K_p} \mb{T}(K_pK^p)$$
where the limit runs over open compact subgroups $K_p$ of $G(\mb{Q}_p)$. We remark that $\mb{T}$ has a finite number of maximal ideals and is a product of its localisation at those maximal ideals. We refer the reader to p. 28 of \cite{so} for details. In particular, if $\mf{m}$ is a maximal ideal of $\mb{T}$, then $\mb{T}_{\mf{m}}$ is a direct factor of $\mb{T}$. 

\medskip

\noindent We define also
$$H^0(K^p, \mc{V}_W) = \varinjlim _{K_p} H^0(S_{K_pK^p} , \mc{V}_W)$$
where $W$ is an irreducible algebraic representation of $G$ and $\mc{V}_W$ is the $E$-local system on $(S_K )_K$ associated to $W$. 

We recall the definition of locally algebraic vectors from \cite{em5}. 

\begin{defi}
Let $G$ be the group of $\mb{Q} _p$-points in some connected linear algebraic group $\mb{G}$ over $\mb{Q} _p$ and let $V$ be a representation of $G$ over $E$. Let $W$ be a finite-dimensional algebraic representation $W$ of $\mb{G}$ over $E$. A vector $v$ in $V$ is locally $W$-algebraic if there exists an open subgroup $H$ of $G$, a natural number $n$, and an $H$-equivariant homomorphism $W^n \ra V$ whose image contains the vector $v$. We write $V_{W-la}$ for the set of locally $W$-algebraic vectors of $V$. 
\end{defi}
Emerton proved in Proposition 4.2.2 of \cite{em5} that $V_{W-la}$ is a $G$-invariant subspace of $V$.
\begin{defi}
A vector $v$ in $V$ is locally algebraic, if it is locally $W$-algebraic for some finite-dimensional algebraic representation $W$ of $\mb{G}$. We denote the set of locally algebraic vectors of $V$ by $V _{l.alg}$.
\end{defi}
It is a $G$-invariant subspace of $V$ by Proposition 4.2.6 of \cite{em5}. We have the following proposition
\begin{prop} \label{prop:alg-vectors}
We have a $G(\mb{A}_{\Sigma _0})$-equivariant isomorphism
$$ \widehat{H}^0(K^p) _{l.alg} \simeq \bigoplus _{W} H^0(K^p, \mc{V}_W) \otimes W ^{\vee}$$
where the sum is taken over all isomorphism classes of irreducible algebraic representations of $G$.
\end{prop}
\begin{proof} This follows from the Emerton spectral sequence. See Corollary 2.2.18 of \cite{em1}.
\end{proof}

\noindent Let $\mf{m}$ be a maximal ideal of $\mb{T}$ which we fix and let $\bar{\rho} _{\mf{m}} : G_F \ra \GL_n (k)$ be the continuous Galois representation which is unramified outside $\Sigma$ and whose characteristic polynomial satisfies
$$\textrm{char}{\bar{\rho} _{\mf{m}}(\Frob _w)} = \sum _{i = 0} ^n (-1)^{n-i} \Nm(w) ^{i(i-1)/2} T_w ^{(i)} X^i \mod  \mf{m}$$
for all places $w$ which do not belong to $\Sigma$ and which split when restricted to $F^+$. This is the Galois representation associated to $\mf{m}$. We refer the reader to Proposition 3.4.2 in \cite{cht} for the construction. We remark that we can suppose that $\bar{\rho} _{\mf{m}}$ is valued in $\GL_n(k)$ after possibly extending $E$ (which we allow).

\medskip

\noindent We assume that the maximal ideal $\mf{m}$ of $\mb{T}$ is non-Eisenstein, that is $\bar{\rho}_{\mf{m}}$ is absolutely irreducible. We let $\rho _{\mf{m}}$ be the universal automorphic deformation of $\bar{\rho} _{\mf{m}}$ over $\mb{T}_{\mf{m}}$ (its construction is standard and we do not recall it here; precise references may be found in Section 4.3 of \cite{cs}). It is an n-dimensional Galois representation over $\mb{T}_{\mf{m}}$ which satisfies 
$$\textrm{char}{\rho _{\mf{m}}(\Frob _w)} = \sum _{i = 0} ^n (-1)^{n-i} \Nm(w) ^{i(i-1)/2} T_w ^{(i)} X^i$$
for all places $w$ which do not belong to $\Sigma$ and which split when restricted to $F^+$.

\section{General formalism}

We now explain the general formalism for proving the pro-modular Fontaine-Mazur conjecture which we specialize at the end to the ordinary setting.


Let $\mb{T} ' _{\mf{m}}$ be a local complete reduced $\mc{O}$-algebra finite over $\mb{T}  _{\mf{m}}$ and let $\rho _{\mf{m}} ' : G_F \ra \GL_n(\mb{T}' _{\mf{m}})$ be the pushout of the universal representation $\rho _{\mf{m}}$ to $\mb{T} _{\mf{m}}'$. In what follows, we will always write $\mf{p}'$ for an ideal of $\mb{T}_{\mf{m}} '$ and $\mf{p}$ for its inverse image in $\mb{T}_{\mf{m}}$. In particular, we will write $\mf{m}'$ for the maximal ideal of $\mb{T}_{\mf{m}}'$.



We will make certain hypotheses (the last one depending on an ideal $\mf{p}' \in \Spec \mb{T} ' _{\mf{m}}$):

\begin{itemize}

\item (H1) There exists an admissible representation $\Pi(\rho ' _{\mf{m},v})$ of $\GL_n(\mb{Q}_p)$ over $\mb{T}_{\mf{m}}'$ associated to each local representation $\rho ' _{\mf{m},v}$ for $v|p$. This representation is such that for each prime ideal $\mf{p}'$ of $\mb{T}_{\mf{m}}'$ which comes from $\Spm(\mb{T}_{\mf{m}}'[1/p])$ (where $\Spm$ is the maximal spectrum, i.e. the set of maximal ideals) for which $\rho' _{\mf{m},v} / \mf{p}'[1/p]$ is regular and de Rham at all places $v$ dividing $p$, the locally algebraic vectors of $\Pi(\rho'_{\mf{m},v})/ \mf{p}' [1/p]$ are non-zero for all $v|p$. Moreover we assume that the $k$-representation $\pi _{\mf{m},v} := \Pi(\rho '_{\mf{m},v})/ \mf{m}'$ is of finite length. 

\item (H2): There exists an allowable set of points for $\Pi(\rho '_{\mf{m},v})$ (for each $v|p$), that is, there exists a dense set of points $\mc{C}$ in $\Spec (\mb{T}_{\mf{m}}')$ which is contained in $\Spm(\mb{T}_{\mf{m}}'[1/p])$ and such that for each $\mf{p}' \in \mc{C}$ we have
$$\Hom _{\mb{T}_{\mf{m}}[G(\mb{Q}_p)]} (\widehat{\otimes} _{v|p} \Pi(\rho '_{\mf{m},v})/ \mf{p}' , \widehat{H}^0(K^p)) \not = 0$$

\item (H3)[$\mf{p}'$]: Every non-zero $\mb{T}_{\mf{m}}[G(\mb{Q}_p)]$-linear map
$$ \widehat{\otimes} _{v|p} \Pi(\rho '_{\mf{m},v})/ \mf{p}' \ra \widehat{H}^0(K^p)$$
is an embedding.

\end{itemize} 

Let us make some comments before showing how these hypotheses imply the pro-modular Fontaine-Mazur conjecture. 

\medskip



\noindent The hypothesis (H1) gives an existence of a representation which shall be viewed as an approximation of the $p$-adic local Langlands correspondence applied to $\rho ' _{\mf{m}}$. In what follows (H1) will be satisfied by using the construction of Breuil-Herzig of the ordinary part of the $p$-adic local Langlands correspondence.

\medskip

\noindent Regarding the hypothesis (H3)[$\mf{p}'$] we will not say anything here. It is needed to deduce that certain locally algebraic vectors are non-zero.

\medskip

\noindent We are left with discussing (H2). Let us define
$$\Pi _p = \widehat{\otimes} _{v|p} \Pi(\rho '_{\mf{m},v})$$
We define $\mb{T}_{\mf{m}} '$-module
$$X = \Hom _{\mb{T}_{\mf{m}}[G(\mb{Q}_p)]} (\Pi _p, \widehat{H}^0(K^p)_{\mc{O}})$$
of $\mathbb{T}  _{\mf{m}}[G(\mb{Q}_p)]$-linear homomorphisms which are $G (\mathbb{Q} _p)$-equivariant and continuous, where $\Pi_p$ is given the $\mf{m}$-adic topology.

The hypothesis (H2) is equivalent to demanding the existence of an allowable set for $\bar{\rho}$ that is a dense subset $\mc{C}$ on $\Spec \mb{T}_{\mf{m}}'$, such that for all $\mf{p}' \in \mc{C}$ we have
$$X [\mf{p}'] = \Hom _{\mb{T}_{\mf{m}}[G(\mb{Q}_p)]} (\Pi _p /\mf{p}'  , \widehat{H}^0(K^p)_{\mc{O}, \mf{m}}) \not = 0$$
Let us prove a preliminary lemma:
\begin{lemm}\label{lem:prem}
$\Hom _{\mc{O}} (X, \mc{O}) \otimes _{\mc{O}} E$ is a finitely generated $\mb{T}_{\mf{m}}'[1/p]$-module.
\end{lemm}
\begin{proof}
By Proposition C.5 of \cite{em2} we have to show that $X$ is cofinitely generated. By Definition C.1, because $\widehat{H}^0(K^p)_{\mc{O}, \mf{m}}$ is $\varpi$-adically complete, separated and $\mc{O}$-torsion free, we are left to show that $(X/\varpi X)[\mf{m}']$ is finite-dimensional over $k$. But we have
$$(X/\varpi X)[\mf{m}'] \hookrightarrow \Hom _{k[G(\mb{Q}_p)]}(\Pi _p / \mf{m}', \widehat{H} ^0 (K^p) _{k, \mf{m}})$$
and we show that $\Hom$ is finite-dimensional. Because $\Pi _p / \mf{m}' = \otimes _{v|p} \pi _{\mf{m},v}$ and each $\pi _{\mf{m},v}$ is of finite length, for each $v$ we can choose a finite-dimensional $k$-subspace $W_v$ of $\pi _{\mf{m},v}$ which generates $\pi _{\mf{m},v}$ as a $\GL_n(\mb{Q}_p)$-representation. Let $W = \otimes _{v|p} W_v$. Since $W_v$ is smooth and finite-dimensional we can choose a compact open subgroup $K_v$ fixing $W_v$ point-wise. Let $K_p = \prod _{v|p} K_v$. By restriction we have
$$\Hom _{k[G(\mb{Q}_p)]}(\Pi _p / \mf{m}', \widehat{H} ^0 (K^p) _{k, \mf{m}}) \hookrightarrow \Hom _{k[K_p]}(W, \widehat{H} ^0 (K^p) _{k, \mf{m}})$$
Since $K_p$ acts trivially on $W$ we moreover have
$$\Hom _{k[K_p]}(W, \widehat{H} ^0 (K^p) _{k, \mf{m}}) \simeq W^{\vee} \otimes _k H^0(S_{K_pK^p},k)_{\mf{m}}$$
which is of finite dimention over $k$.
\end{proof}

\begin{lemm}\label{lem:allowable}
Assume (H2). Then $X[\mf{p}'] \not = 0$ for all $\mf{p}' \in \Spec \mb{T}_{\mf{m}}'$.
\end{lemm}
\begin{proof}
By Lemma C.14 of \cite{em2}, we have
$$(\mathbb{T}  _{\mf{m}} ' / \mf{p}') \otimes _{\mathbb{T}  _{\mf{m}} '} \Hom _{\mc{O}} (X, \mc{O}) \otimes _{\mc{O}} E \simeq \Hom _{\mc{O}} (X[\mf{p}'], \mc{O}) \otimes _{\mc{O}} E$$
and so it suffices to show that the elements on the right are non-zero for all $\mathfrak{p}'$ if and only if they are non-zero for all $\mathfrak{p}'$ in $\mc{C}$. Consider things in more generality. Let $M$ be a finitely generated $\mb{T}_{\mf{m}}'[1/p]$-module such that $M/\mf{p}'M \not =0$ for all $\mf{p}' \in \mc{C}$. Because $\mb{T}_{\mf{m}}'[1/p] / \mf{p}'$ is a field, it follows that $M/ \mf{p}' M$ is a faithful $\mb{T}_{\mf{m}}'[1/p] /\mf{p}'$-module. If $t \in \mb{T}_{\mf{m}}'[1/p]$ acts by $0$ on $M$ then it acts by $0$ on $M/\mathfrak{p}' M$ for all $\mf{p}'$, and as $M / \mf{p}' M \not =0$ if $\mf{p}' \in \mc{C}$, we have $t\in \mf{p}'$ for all $\mf{p}' \in \mc{C}$, that is $t \in \cap _{\mf{p}' \in \mc{C}} \mf{p}' = 0$. So $\mb{T}_{\mf{m}}'[1/p]$ acts faithfully on $M$. Now, let $\mf{p}'$ be any maximal ideal of $\mb{T}_{\mf{m}}'[1/p]$ and suppose that $M / \mf{p}' M = 0$, that is $M = \mf{p}' M$. As $M$ is finitely generated $\mb{T}_{\mf{m}}'[1/p]$-module, Nakayama's lemma gives us a non-zero element $t$ of $\mathbb{T}  _{\mf{m}}'[1/p]$ such that $tM =0$, which is impossible as we have shown above. We deduce that $M / \mf{p}' M \not =0$ for all $\mf{p}'$. Applying this reasoning to $M = \Hom _{\mc{O}} (X, \mc{O}) \otimes _{\mc{O}} E$ which is finitely generated by Lemma \ref{lem:prem}, we conclude.
\end{proof}

\medskip

\begin{defi}
We say that a representation $\rho: \Gal(\bar{F}/F) \ra \GL_n(E)$ is pro-modular with respect to $\mb{T}_{\mf{m}}'$ if there exists a prime ideal $\mf{p}'$ of $\mb{T}_{\mf{m}} '$ such that $\rho \simeq \rho _{\mf{m}} / \mf{p}[1/p]$ and $\widehat{H} ^0 (K^p)[\mf{p}] \not = 0$, where $\mf{p}$ is the inverse image of $\mf{p}'$ in $\mb{T}_{\mf{m}}$. 
\end{defi}
One natural source of pro-modular representations are representations attached to points on the eigenvariety for $G$. We shall review this notion later on.

We say that $\rho$ is modular if it is the Galois representation associated to some automorphic representation of $G$ of tame level $K^p$. This is equivalent to $\widehat{H} ^0 (K^p)_{l.alg}[\mf{p}] \not = 0$ by 
Proposition \ref{prop:alg-vectors}. Our three hypotheses imply the pro-modular Fontaine-Mazur conjecture in the following form.

\begin{theo}\label{thm:FM}
Let $\rho: \Gal(\bar{F}/F) \ra \GL_n(E)$ be a pro-modular Galois representation with respect to $\mb{T}_{\mf{m}}'$ with the associated prime ideal $\mf{p}'$ of $\mb{T}_{\mf{m}}'$. Assume that $\rho$ is de Rham and regular at all places dividing $p$. Assume also that hypotheses (H1),(H2) and (H3)[$\mf{p}$'] hold. Then $\rho$ is modular.
\end{theo}
\begin{proof}
As $\rho _{v}$ is de Rham and regular for every $v|p$, by (H1) we have that $\Pi(\rho _{v}) _{l.alg} \not = 0$ for every $v|p$. By Lemma \ref{lem:allowable} and the hypothesis (H3)$[\mf{p}']$ we conclude that also 
$$\widehat{H} ^0 (K^p)_{l.alg}[\mf{p}] \not = 0$$
which is what we wanted.
\end{proof}

In the rest of this text we will explain the ordinary setting.

\section{Ordinary case}

In this section, we show that the ordinary part of Breuil-Herzig (\cite{bh}) fulfills the formalism presented in the previous section.

\subsection{Preliminaries on reductive groups}

We recall certain results on reductive groups used in \cite{bh}. Let $G$ be a split connected reductive $\mb{Z}_p$-group with a Borel subgroup $B$ and a torus $T \subset B$. We let $(X(T), R, X^{\vee}(T), R^{\vee})$ be the root datum of $G$, where $R \subset X(T)$ (respectively $R^{\vee} \subset X^{\vee}(T)$) is the set of roots (resp. coroots). For each $\alpha \in R$, let $s_{\alpha}$ be the reflection on $X(T)$ associated to $\alpha$. Let $W$ be the Weyl group, the subgroup of automorphisms of $X(T)$ generated by $s_{\alpha}$ for $\alpha \in R$.

\medskip

We fix a subset of simple roots $S \subset R$ and we let $R^+ \subset R$ be the set of positive roots (roots in $\oplus _{\alpha \in S} \mb{Z} _{\geq 0} \alpha$). Let $G^{der}$ be the derived group of $G$ and let $\hat{G}$ be the dual group scheme of $G$ (which we get by taking the dual root datum). We have also dual groups $\hat{B}$ and $\hat{T}$.

\medskip

To $\alpha \in R$ one can associate a root subgroup $U_{\alpha} \subset G$. We have $\alpha \in R^+$ if and only if $U_{\alpha} \subset B$. We let $\mf{g}_{\alpha}$ be the Lie algebra of $U_{\alpha}$. We call a subset $C \subset R$ closed if for each $\alpha \in C, \beta \in C$ such that $\alpha + \beta \in R$, we have $\alpha +\beta \in C$. If $C \subset R^+$ is a closed subset, we let $U_C \subset U$ be the Zariski closed subgroup of $B$ generated by the root subgroups $U_{\alpha}$ for $\alpha \in C$. We let $B_C = T U_C$ be the Zariski closed subgroup of $B$ determined by $C$.

\medskip



Let us spell out all the assumptions that we put on $G$ and its dual group $\hat{G}$. We suppose throughout this text that both $G$ and $\hat{G}$ have connected centers. Moreover we suppose that $G^{der}$ is simply connected (some of these conditions are equivalent, see Proposition 2.1.1 in \cite{bh}). This condition implies that there exist fundamental weights $\lambda _{\alpha}$ for $\alpha \in S$. They satisfy for any $\beta \in S$
$$ \left< \lambda _{\alpha}, \beta ^{\vee} \right> = \left\{
  \begin{array}{l l}
    1 & \quad \textrm{ if }\alpha = \beta\\
    0 & \quad \text{ if }\alpha \not = \beta
  \end{array} \right. $$  
We define as in Section 3.1 of \cite{bh} a twisting element $\theta$ for $G$ by setting $\theta = \sum _{\alpha \in S} \lambda _{\alpha}$. For any $\alpha \in S$ we have $\left< \theta , \alpha ^{\vee} \right> = 1$. 

\medskip

If $C \subset R$ is a closed subset, write $G_C$ for the Zariski closed subgroup scheme of $G$ generated by $T$, $U_{\alpha}$ and $U_{-\alpha}$ for $\alpha \in C$. For $C = \{\alpha\}$ we write simply $G_{\alpha}$ for $G_C$. A subset $J\subset S$ of pairwise orthogonal roots is closed (see the proof of Lemma 2.3.7 in \cite{bh}) and hence we can define $G_J$ as above.

\begin{lemm}[Lemma 3.1.4, \cite{bh}] Let $J \subset S$ be a subset of pairwise orthogonal roots. Then there is a subtorus $T' _J \subset T$ which is central in $G_J$ such that
$$G_J \simeq T' _J \times \GL_2 ^J$$
\end{lemm}

We use this lemma in the construction of $\Pi(\rho)^{ord}$ which we define as a sum over certain induced representations of $G_J(\mb{Q}_p)$. We construct representations of $G_J(\mb{Q}_p)$ by using the $p$-adic local Langlands correspondence for $\GL_2(\mb{Q}_p)$.

\subsection{Ordinary part of the p-adic local Langlands correspondence}

Let $E$ be a finite extension of $\mb{Q}_p$ with ring of integers $\mc{O}$ and let $k$ be its residue field. We fix also a uniformiser $\varpi$. Let $A$ be a complete local Noetherian $\mc{O}$-algebra with residue field $k$.

\medskip

We have
$$T(\mb{Q}_p) = \Hom _{\Spec (\mb{Q}_p)}(\Spec(\mb{Q}_p), \Spec (\mb{Q}_p[X(T)])) = \Hom _{\mb{Z}} (X(T), \mb{Q}_p ^{\times}) = $$
$$= \Hom _{\mb{Z}}(X(T), \mb{Z}) \otimes _{\mb{Z}} \mb{Q}_p ^{\times} = X(\hat{T}) \otimes _{\mb{Z}} \mb{Q}_p ^{\times}$$
To a continuous character
$$\hat{\chi} : \Gal(\bar{\mb{Q}}_p / \mb{Q}_p) \twoheadrightarrow \Gal(\bar{\mb{Q}}_p / \mb{Q}_p)^{ab} \ra \hat{T}(A)$$
we can associate a continuous character $\chi: T(\mb{Q}_p) \ra A^{\times}$ by taking the composite of the maps
$$T(\mb{Q}_p) \simeq X(\hat{T}) \otimes _{\mb{Z}} \mb{Q}_p ^{\times} \hookrightarrow X(\hat{T}) \otimes _{\mb{Z}} \Gal(\bar{\mb{Q}}_p / \mb{Q}_p)^{ab} \ra X(\hat{T}) \otimes _{\mb{Z}} \hat{T} (A) \ra A^{\times}$$
where the first injection comes from the local class field theory.

\medskip

We define the $p$-adic cyclotomic character $\epsilon : G_{\mb{Q}_p} \ra A^{\times}$ by composing the standard $p$-adic cyclotomic character which takes values in $\mc{O} ^{\times}$ with the inclusion $\mc{O} ^{\times} \hookrightarrow A ^{\times}$. By the local class field theory we can also consider it as a character of $\mb{Q}_p ^{\times}$ which we tacitly do in what follows.

\medskip

Let us consider a continuous homomorphism
$$\rho : \Gal (\bar{\mb{Q}} _p / \mb{Q} _p) \ra \hat{G} (A)$$
\begin{defi}
We say that $\rho$ is \textbf{triangular} when it takes values in our fixed Borel $\hat{B}(A)$ of $\hat{G}(A)$. 
\end{defi}
We let $C_{\rho} \subset R ^{+ \vee}$ be the smallest closed subset such that $\hat{B} _{C_{\rho}}(A)$ contains all the $\rho (g)$ for $g \in \Gal(\bar{\mb{Q}}_p / \mb{Q}_p)$ (compare with Lemma 2.3.1 of \cite{bh}). Thus $\rho$ factorises
 via $\hat{B}_{C_{\rho}}(A)$
$$\rho :  \Gal (\bar{\mb{Q}} _p / \mb{Q} _p) \ra \hat{B} _{C_{\rho}}(A) \subset \hat{B}(A) \subset \hat{G} (A)$$
We associate a character $\hat{\chi} _{\rho}$ to $\rho$ by composing $\rho$ with the natural surjection
$$\hat{\chi} _{\rho} : \Gal(\bar{\mb{Q}}_p / \mb{Q}_p) \ra \hat{B}_{C_{\rho}}(A) \twoheadrightarrow \hat{T}(A)$$
We attach to $\hat{\chi}_{\rho}$ a continuous character $\chi _{\rho} : T(\mb{Q}_p) \ra A ^{\times}$ by the local class field theory as above.

\begin{defi} We say that a triangular $\rho$ is \textbf{generic} if $\alpha ^{\vee} \circ \hat{\chi} _{\rho} \notin \{1,\epsilon, \epsilon ^{-1}\}$ for all $\alpha \in R^+$ (or equivalently all $\alpha \in R$). The same definition applies to the reduction $\bar{\rho}$ of $\rho$.




\end{defi}

In what follows we will consider only triangular representations $\rho$. We assume that $\bar{\rho}$ is generic.

We now construct several representations of $G(\mb{Q}_p)$ over $A$ attached to $\rho$. Let $I \subset S^{\vee}$ be a subset of pairwise orthogonal roots. We shall firstly construct an admissible continuous representation $\tilde{\Pi}(\rho)_I$ of $G_{I ^{\vee}}(\mb{Q}_p)$ over $A$. We imitate the proof of Proposition 3.3.3 in \cite{bh}, though we present a simplified construction, because we do not need to show unicity of $\tilde{\Pi}(\rho)_I$. Only later on and under additional assumptions we will show that we retrieve the construction of Breuil and Herzig over fields. Then we obtain a representation $\Pi(\rho)^{ord}$ of $G(\mb{Q}_p)$ over $A$, which generalizes the construction of Breuil and Herzig over fields, and which we define as a direct limit of $\Pi(\rho)_I$ over different $I$ (where $\Pi(\rho)_I$ is simply $\tilde{\Pi}(\rho)_I$ induced to $G(\mb{Q}_p)$). In particular, we shall consider a representation $\Pi(\rho)_{\emptyset} = (\Ind _{B(\mb{Q}_p)} ^{G(\mb{Q}_p)} \chi _{\rho}  \cdot (\epsilon ^{-1} \circ \theta))^{\mc{C}^0}$ which we use for the proof of the pro-modular Fontaine-Mazur conjecture. All these representations are functorial in $A$ and hence behave well with respect to reduction modulo prime ideals.



\medskip



If $\beta \in I ^{\vee}$ and $\chi _{\beta} : T _{\beta} (\mb{Q}_p) \ra A^{\times}$ is a continuous character, we define 
$$\Pi _{\beta}(\chi _{\beta}) = \left( \Ind ^{\GL_2(\mb{Q}_p)} _{\left(\begin{smallmatrix} * & 0 \\ * & * \end{smallmatrix}\right)} \chi _{\beta} \cdot (\epsilon ^{-1} \circ \theta) _{| T_{\beta}(\mb{Q}_p)} \right) ^{\mc{C} ^0}$$
This is a representation of $\GL_2(\mb{Q}_p)$ which we use as a building block. We let $\rho _{\beta}: \Gal(\bar{\mb{Q}}_p / \mb{Q}_p) \ra \GL_{2,\beta ^{\vee}}(A)$ be the representation which we get by composing $\rho: \Gal(\bar{\mb{Q}}_p / \mb{Q}_p) \ra \hat{B}(A)$ with $\hat{B}(A) \ra \hat{B}_{\beta}(A) \ra \GL_{2,\beta ^{\vee}}(A)$. We define $\mc{E}_{\beta}$ as the representation attached to the 2-dimensional Galois representation $\rho _{\beta}$ by the p-adic local Langlands correspondence for $\GL_2(\mb{Q}_p)$. In order to have a functorial construction we fix a quasi-inverse $\MF ^{-1}$ to the Colmez functor $\MF$ for $\GL_2(\mb{Q}_p)$ (we use the notation of Emerton from \cite{em2}). For any $\beta$ it sends a lifting of $\bar{\rho} _{\beta}$ to a lifting of $\bar{\pi} _{\beta}$ with a central character (where $\bar{\pi} _{\beta}$ is the smooth representation over $k$ corresponding to $\bar{\rho} _{\beta}$ by the mod $p$ local Langlands correspondence). Then we define
$$\mc{E}_{\beta} = \MF ^{-1}(\rho _{\beta})$$
We remark that over $k$ this is an extension of $\Pi _{\beta}(s_{\beta}(\chi _{\rho|T_{\beta}(\mb{Q}_p)}))$ by $\Pi _{\beta}(\chi _{\rho |T_{\beta}(\mb{Q}_p)})$ because $\rho _{\beta}$ is lower-triangular with the appropiate character on the diagonal (see Proposition 3.4.2 in \cite{em2}). 

\medskip

Let $\chi _{\rho, I ^{\vee}} ' = \chi _{\rho |T_{I ^{\vee}}' (\mb{Q}_p)}$. We define an admissible continuous representation of $T ' _{I ^{\vee}}(\mb{Q}_p) \times \GL_2(\mb{Q}_p) ^{I ^{\vee}}$
$$\tilde{\Pi}(\rho)_I = \chi _{\rho, I ^{\vee}} ' \cdot (\epsilon ^{-1} \circ \theta _{|T ' _{I ^{\vee}}(\mb{Q}_p)}) \otimes _{A} (\widehat{\otimes} _{\beta \in I ^{\vee}} \mc{E} _{\beta})$$
This is exactly the representation we look for. 

\medskip

We set
$$\Pi(\rho)_I = \left( \Ind ^{G(\mb{Q}_p)} _{B^-(\mb{Q}_p)G_{I ^{\vee}}(\mb{Q}_p)} \tilde{\Pi}(\rho)_I \right)^{\mc{C}^0}$$
where we view $\tilde{\Pi}(\rho)_I$ as a continuous representation of $B^-(\mb{Q}_p)G_{I ^{\vee}}(\mb{Q}_p)$ by inflation. By the proposition above and by Theorem 3.1.1 in \cite{bh} (which holds in our setting verbatim), the representation $\Pi(\rho)_I$ of $G(\mb{Q}_p)$ is admissible and continuous. 

\medskip

We now use an argument similar to the one of Breuil-Herzig appearing before Lemma 3.3.5 in \cite{bh} to construct a direct limit. Following the proof of Proposition 3.4.2 of \cite{em2} we have natural injections of $\Pi _{\beta} (\chi _{\rho |T_{\beta}(\mb{Q}_p)})$ into $\mc{E}_{\beta}$. Indeed, Proposition 3.2.4 of \cite{em2} gives us a natural embedding $\chi _{\rho |T_{\beta}(\mb{Q}_p)} \hookrightarrow \Ord (\mc{E}_{\beta})$, where we have denoted by $\Ord$ the ordinary part functor of Emerton. By adjointness property of $\Ord$ this gives us a $\GL_2(\mb{Q}_p)$-equivariant injection  $\Pi _{\beta} (\chi _{\rho |T_{\beta}(\mb{Q}_p)}) \hookrightarrow \mc{E}_{\beta}$. We remark that those injections will be functorial because of Proposition 3.2.4 of \cite{em2} and because we have fixed a quasi-inverse $\MF ^{-1}$.


By Theorem 4.4.6 and Corollary 4.3.5 of \cite{em3}, we have for $I' \subset I$
$$\Hom _{G(\mb{Q}_p)}(\Pi(\rho)_{I'}, \Pi(\rho)_I) \simeq \Hom _{G_{I ^{\vee}}(\mb{Q}_p)}\left( (\Ind ^{G_{I ^{\vee}}(\mb{Q}_p)} _{(B^{-}(\mb{Q}_p)\cap G_{I ^{\vee}}(\mb{Q}_p))G_{I ^{' \vee}}(\mb{Q}_p)} \tilde{\Pi}(\rho)_{I'})^{C^0}, \tilde{\Pi}(\rho)_I \right)$$
Observe that our injections 
$$\Pi _{\beta} (\chi _{\rho |T_{\beta}(\mb{Q}_p)}) \hookrightarrow \mc{E}_{\beta}$$
invoked above induce an injection
$$\Ind ^{G_{I ^{\vee}}(\mb{Q}_p)} _{(B^{-}(\mb{Q}_p)\cap G_{I ^{\vee}}(\mb{Q}_p))G_{I ^{' \vee}}(\mb{Q}_p)} ( \tilde{\Pi}(\rho)_{I'})^{C^0} \hookrightarrow \tilde{\Pi}(\rho)_I$$
and hence also a $G(\mb{Q}_p)$-equivariant injection
$$\Pi(\rho)_{I'} \hookrightarrow \Pi(\rho)_I$$
This actually gives a compatible system of injections, by which we mean that for any $I'' \subset I' \subset I$, the corresponding diagram of injections is commutative. We then define an admissible continuous representation of $G(\mb{Q}_p)$ over $A$ by
$$\Pi(\rho)^{ord} = \varinjlim _I \Pi(\rho)_I$$
where $I$ runs over subsets of $S^{\vee}$ of pairwise orthogonal roots. 






\subsection{Compatibility with the construction of Breuil-Herzig}

We study in this section how $\Pi(\rho)^{ord}$ behaves with respect to reduction modulo prime ideals in $A$. Recall that $G$ and its dual are split, hence we can canonically identify $R ^{\vee}(A)$ and $R ^{\vee}(A/\mf{p})$ for any prime ideal $\mf{p}$ of $A$.

\begin{lemm}\label{lem:reduction}
Let $A \ra A'$ be a morphism of complete local $\mc{O}$-algebras and let $\rho$ be triangular over $A$ with $\bar{\rho}$ generic. Then 
$$\Pi(\rho \otimes _A A')_I \simeq \Pi(\rho)_I \otimes_A A'$$ 
for any subset $I \subset  S^{\vee}$ of pairwise orthogonal roots and
$$\Pi(\rho\otimes_A A')^{\ord} \simeq \Pi(\rho)^{\ord} \otimes_A A'$$
\end{lemm}
\begin{proof}
Observe that $\rho \otimes _A A'$ is triangular because $\rho$ is. By the definition of $\Pi(\rho) _I$ it is enough to check, that the construction of $\tilde{\Pi}(\rho)_I$ we have given above is compatible with the base change $A \ra A'$. This follows from the fact that the p-adic local Langlands correspondence for $\GL_2(\mb{Q}_p)$ is compatible with the base change $A \ra A'$. 

\end{proof}







To put more content into this lemma let us specialize to the totally indecomposable case.

\begin{defi}
We say that $\rho$ is \textbf{totally indecomposable} if $C_{\rho} = R^{+ \vee}$ is minimal among all conjugates of $\rho$ by $B$ (equivalently, $C_{b\rho b^{-1}} = R^{+ \vee}$ for all $b\in B$).
\end{defi}




\medskip

We prove now that for $\GL_n$ we retrieve the construction of Breuil-Herzig after reducing modulo $\mf{p}$. Before continuing, we shall give another characterisation of totally indecomposable representations valable for $G = \GL_n$.

\begin{lemm}\label{lem:tot-indec}
Let $\rho: \Gal(\bar{\mb{Q}}_p / \mb{Q}_p) \ra \GL_n(A)$ be a triangular representation and $A$ be a field. The following conditions are equivalent:
\begin{enumerate}
\item All semi-simple subquotients of $\rho$ are simple (equivalently, the graded pieces of the filtration by the socle are irreducible).

\item $B$ is the unique Borel that contains the image of $\rho$  (equivalently, the image of $\rho$ fixes a unique Borel $B$ (flag)). Here $B$ is the Borel we have fixed before in the definition of being triangular.

\item $\rho$ is totally indecomposable.
\end{enumerate}
\end{lemm}
\begin{proof}
(1. $\Leftrightarrow$ 2.) If there exists $\soc _{j+1} / \soc _j$ which is not irreducible then we can construct two distinct flags which are stable by the image of $\rho$. On the other hand, if there exists two distinct flags fixed by the image of $\rho$, say 
$$ V_1 \subset V_2 \subset ... \subset V_n$$
and
$$ V_1 ' \subset V_2 ' \subset ... \subset V_n '$$
and let $j$ be the smallest index such that $V_j \not = V_j '$. Then $(V_j + V_j ') / V_{j-1}$ is of dimension 2 and semi-simple, hence $\rho$ is not totally indecomposable.  

\medskip

\noindent (2. $\Leftrightarrow$ 3.) Suppose that $\rho$ stabilizes another Borel $B'$ (apart from $B$). Let $b\in B$ be an element which conjugates $B'$ into a Borel containing the maximal torus $T$. This Borel $bB' b^{-1}$ is of the form $w(B)$ for some $w$ in the Weyl group. Hence we see that $C_{b\rho b^{-1}}$ is contained in the intersection of $R^{+ \vee}$ and $w(R^{+ \vee})$, and in particular is different from $R^{+ \vee}$.

If $C_{\rho}$ is different from  $R^{+ \vee}$, then there exists a positive simple root $\alpha$ which does not belong to $C_{\rho}$. It follows that $s_{\alpha}(C_{\rho})$ is contained in $R^{+ \vee}$ and hence the image of $\rho$ is contained in $s_{\alpha}(B)$.
\end{proof}

\begin{lemm}\label{lem:tot-reduction}
Let $\rho: \Gal(\bar{\mb{Q}}_p / \mb{Q}_p) \ra \GL_n(\mc{O})$ be a triangular representation such that $\bar{\rho}$ is triangular, generic and totally indecomposable. Then $\rho _E = \rho \otimes _{\mc{O}} E$ is also totally indecomposable and generic.
\end{lemm}
\begin{proof}
The statement about genericity of $\rho _E$ is clear. Let us prove that it is totally indecomposable. Let us denote by $\bar{\chi} _j$ characters appearing on the diagonal of $\bar{\rho}$ which we have supposed to be pairwise distinct hence linearly independent. Let $B$ be a Borel in $\GL_n(E)$ containing the image of $\rho$. It corresponds to a flag
$$V_1 \subset V_2 \subset ... \subset V_n = E^n$$
By intersection with $\mc{O}^n$ we obtain a flag
$$\omega _1 \subset \omega _2 \subset ... \subset \omega _n = \mc{O} ^n$$
of $\mc{O} ^n$ stable by the image of $\rho$ which reduces to the standard flag modulo $\mf{m}$ by the hypothesis that $\bar{\rho}$ is totally indecomposable. In particular, we see that $G$ acts on $V_i / V_{i-1}$ by a character $\chi _i$ with values in $\mc{O}^{\times}$ which lifts the character $\bar{\chi} _i$. By genericity of $\bar{\rho}$, the characters $\chi _i$ are mutually distinct and each appears in the semi-simplification of $\rho$ with multiplicity $1$.

Suppose now that we have another Borel $B'$ different from $B$ and stable by the image of $\rho$ with the associated flag
$$V_1 ' \subset V_2 ' \subset ... \subset V_n '$$
Let $i$ be the smallest index $i$ such that $V_i ' \not = V_i$. Then $G$ acts on the $2$-dimensional subquotient $(V_i + V_i ') / V_{i-1}$ by the character $\chi _i$, which contradicts the fact that $\chi _i$ appears with multiplicity $1$. Hence $B' =B$ and we see that $\rho$ is totally indecomposable by Lemma \ref{lem:tot-indec}.
\end{proof}

\begin{prop}\label{prop:comp}
Suppose that $\bar{\rho}$ is generic, triangular and totally indecomposable and $\rho$ is triangular. Then for any morphism $A \ra E'$ (where $E'$ is a finite extension of $E$), the $E'$-Banach representation $\Pi(\rho)^{\ord} \otimes _A E'$ is the representation $\Pi(\rho \otimes_A E')^{\ord}$ of Breuil and Herzig.
\end{prop}
\begin{proof}
By Lemma \ref{lem:reduction} we can suppose that $A = \mc{O} _{E'}$. Observe that $\rho _{E'} = \rho \otimes _{\mc{O}_{E'}} E'$ is generic and totally indecomposable by Lemma \ref{lem:tot-reduction}. To finish the proof we have to show that $\rho _{E'}$ is a good conjugate of itself (Definition 3.2.4 in \cite{bh}). This follows from (3) of Lemma \ref{lem:tot-indec} and we conclude by Lemma 3.3.5 of \cite{bh}.
\end{proof}


\subsection{Universal ordinary modular representation}

In this subsection we will apply the formalism developed above to a particular example. We consider triangular deformations of modular representations and our goal is to define $\Pi(\rho _{\mf{m}, w}) ^{ord}$, where $\rho _{\mf{m}, w}$ is a certain universal modular Galois representation at a place $w|p$.

\medskip

We take up the setting of Section 2. For each place $w|p$ of $F^+$ we choose a place $\tilde w$ of $F$, so as to get an identification 
$$G(\mb{Q}_p)\simeq \prod_{w|p} \GL_n(F_{\tilde w})=  \GL_n(\mb{Q}_p)^f$$ 
where $f=[F^+:\mb{Q}]$. We denote by $B$ the upper triangular Borel subgroup and we have $B\simeq B_n(\mb{Q}_p)^f$.

\medskip

We will now define a certain quotient $\mb{T}(K^p)^{\ord}$ of $\mb{T}(K^p)$. There are two equivalent approaches for this. 

\medskip

Firstly, we may follow Geraghty who introduced in 2.4 \cite{ge} a certain direct factor $\mb{T}(K^pK_p(n))^{\ord}$ of $\mb{T}(K^p K_p(n))$ where $K_p(n)$ denotes the group of matrices in $\GL_n(\mb{Z}_p)^f$ that reduce to a unipotent upper-triangular matrix mod $p^n$. More precisely, in loc. cit is defined the algebra $\mb{T}_{\lambda} ^{T,\ord}(U(l^{n,n}),\mc{O})$. There, $\lambda$ is a dominant weight for $G$ (but we take $\lambda=0$ in this case), $U(l^{n,n})$ is our $K^pK_p(n)$ (our $p$ is denoted by $l$), $T$ is our $\Sigma$. Beware that Geraghty's algebra contains diamond operators at places above $p$ (his $l$), in contrast with ours. So our $\mb{T}(K^pK_p(n))^{\ord}$ is the image of $\mb{T}(K^p)$ in Geraghty's $\mb{T}_0 ^{T,\ord}(U(l^{n,n}),\mc{O})$. When $n$ varies, these constructions are compatible and we may take the projective limit $\mb{T}^{T,\ord}_0(U(l^\infty),\mc{O})$. We get a quotient $\mb{T}(K^p)^{\ord}$ as the image of the natural map $\mb{T}(K^p) \ra \mb{T}^{T,\ord}_0(U(l^\infty),\mc{O})$ in Geraghty's notation on p.14 of loc. cit. 

\medskip

Alternatively, we may use Emerton's ordinary part functor and define 
$$\mb{T}(K^p)^{\ord} := \textrm{image of }\mb{T}(K^p)\textrm{ in }\End_{\mc{O}} (\Ord _B (\widehat{H}^0(K^p)))$$
Note that $\Ord_B( \widehat{H}^0(K^p))$ is a continuous representation of $T(\mb{Q}_p)$ over $\mb{T}(K^p)$ and in particular is a $\mb{T}(K^p)[[T(\mb{Z}_p)]]$-module.
Then Geraghty's algebra can be identified with the image of $\mb{T}(K^p)[[T(\mb{Z}_p)]]$ in $\End_{\mc{O}} (\Ord_B (\widehat{H}^0(K^p)))$ (compare with 5.6 of \cite{em2}).

\begin{lemm}
$\mb{T}(K^p)^{\ord}$ is a direct factor of $\mb{T}(K^p)$.
\end{lemm}
\begin{proof} 
Denote by $\mb{T}(K^p N(\mb{Z}_p))$ the image of $\mb{T}(K^p)$ in $\End_{\mc{O}}(\widehat{H}^0(K^p)^{N(\mb{Z}_p)})$. Note that in Geraghty's notation, $\mb{T}(K^p N(\mb{Z}_p))$ is the image of $\mb{T}(K^p)$ in $\mb{T}^T(U(l^\infty),\mc{O})$. By definition $\mb{T}(K^p)^{\ord}$ is thus a direct factor of $\mb{T}(K^p N(\mb{Z}_p))$. Therefore, we need to show that $\mb{T}(K^p) = T(K^p N(\mb{Z}_p))$. This amounts to prove that $\mb{T}(K^p)$ acts faithfully on $\widehat{H}^0(K^p)^{N(\mb{Z}_p)}$. Since $\mb{T}(K^p)$ is reduced and $\mc{O}$-torsion free, it suffices to prove that $\widehat{H} ^0(K^p)^{N(\mb{Z}_p)}_E[\mf{p}]$ is non-zero for a dense set of primes $\mf{p}$ in $\Spec(\mb{T}(K^p)[1/p])$. The set $P_{alg}$ of prime ideals $\mf{p}$ in $\mb{T}(K^p)_E$ such that $\widehat{H}^0(K^p)_E[\mf{p}]_{l.alg}$ is non-zero is known to be dense (Corollary 4 in \cite{so}). Now, for any irreducible locally algebraic representation $\pi \otimes W$ of $G(\mb{Q}_p)$, we have that $(\pi \otimes W)^{N(\mb{Z}_p)}$ is non-zero. Therefore, for all $\mf{p} \in P_{alg}$ we have $\widehat{H}^0(K^p)^{N(\mb{Z}_p)}_E[\mf{p}] \not = 0$ and we conclude.
\end{proof}

Let now $\mf{m}$ be a maximal ideal of $\mb{T}(K^p)$ as in Section 2. We say that $\mf{m}$ is \textbf{ordinary} if it comes from a maximal ideal of $\mb{T}(K^p)^{\ord}$. By the lemma above, the quotient map $\mb{T}(K^p)_{\mf{m}} \ra \mb{T}(K^p)^{\ord} _{\mf{m}}$ is an isomorphism.

\medskip

Let us fix an ordinary non-Eisenstein maximal ideal $\mf{m}$ of $\mb{T}(K^p)$. Recall that we have defined $\bar\rho_{\mf{m}}$ and $\rho_{\mf{m}}$ in Section 2. For any prime ideal $\mf{p}$ in $\mb{T}(K^p)_{\mf{m}}$ coming from the maximal spectrum $\Spm (\mb{T}(K^p)_{\mf{m}}[1/p])$, we will write 
$$\rho_{\mf{p}} := \rho_{\mf{m}} \otimes_{\mb{T}(K^p)_{\mf{m}}} \mb{T}(K^p)_{\mf{m}}/\mf{p} [1/p]$$
which is a continuous Galois representation over a finite extension of $\mb{Q}_p$. We will need the following result of Geraghty:

\begin{prop}\label{prop:dense}
Consider the set $P_{autom} ^{cris}$ of maximal ideals in $\mb{T}(K^p)_{\mf{m}}[1/p]$ such that $H^0(K^p K_p(0),\mc{V}_W)[\mf{p}]$ is non-zero for some irreducible algebraic representation $W$ of $G$. Then:

\begin{itemize}

\item This set is Zariski dense in $\Spec(\mb{T}(K^p)_{\mf{m}}[1/p])$.
 
\item For any $\mf{p}$ in $P_{autom} ^{cris}$, the representation $\rho_{\mf{p}}$ is triangularisable (and crystalline) at each place dividing $p$.
\end{itemize}
\end{prop}
\begin{proof}
The point 2. follows from Corollary 2.7.8 of \cite{ge}. The point 1. is the density of cristalline points which is proved in Corollary 4 of \cite{so}, or can be deduced from the density result of Hida used by Geraghty in the proof of Corollary 3.1.4 in \cite{ge}.
\end{proof}

As a consequence of this proposition, the residual representation $\bar\rho_{\mf{m},w}$ is triangularisable for each $w|p$.

\medskip

We now assume further that $\bar\rho_{\mf{m},w}$ is totally indecomposable and generic for each $w|p$. Note that generic was only defined for triangular representations.
However the definition extends unambiguously to triangularisable representations, provided they are totally indecomposable, because such representations factor through a unique Borel subgroup.

\medskip

Our goal is to define $\Pi(\rho _{\mf{m}, w}) ^{ord}$ where $\rho _{\mf{m}, w}$ is the restriction of $\rho _{\mf{m}}$ to the decomposition group $G_{F_w} = \Gal (\bar{F} _w / F_w)$ for any place $w|p$ of $F$. In order to do so, we need to prove that $\rho _{\mf{m}, w}$ is triangularisable. We basically do so, but not over $\mb{T}_{\mf{m}}$ but rather over a bigger $\mc{O}$-algebra $\mb{T} _{\mf{m}} '$. This is sufficient for our applications. 

\medskip



Following Geraghty (Section 3.1 of \cite{ge}) we introduce a subfunctor $G$ of $\Spec \mb{T} _{\mf{m}} \times \mc{F}$ defined on $A$-points as the set of $\mc{O}$-homomorphisms $\mb{T}_m \ra A$ and filtrations $\Fil \in \mc{F} (A)$ ($\mc{F}$ is the flag variety) preserved by the induced representation $\rho _{A,w}$. In fact, Geraghty defined this functor over a universal ring $R$, but we shall need it only over the Hecke algebra.


\medskip

This functor is representable by a closed subscheme $\mc{G}$ of $\Spec \mb{T}_{\mf{m}} \times \mc{F}$ (Lemma 3.1.2 in \cite{ge}). We consider the resulting morphism $f: \mc{G} \ra \Spec \mb{T} _{\mf{m}}$. 

\begin{prop}
The morphism $f: \mc{G} \ra \Spec \mb{T} _{\mf{m}}$ is proper with geometric fibres of cardinal one.
\end{prop}
\begin{proof}
The properness of $f$ follows from that of the flag variety (cf. the proof of Lemma 3.1.3 in \cite{ge}). Let us now prove that each geometric fibre is of cardinal one. Let us denote by $\bar{\chi}_{1,w},\bar{\chi}_{2,w},..., \bar{\chi}_{n,w}$ characters of $G_{F_w}$ appearing on the diagonal of $\bar{\rho}_{\mf{m},w}$. Firstly, we remark that geometric fibres are non-empty. Indeed, $f$ is dominant by Proposition \ref{prop:dense}, hence surjective since it is proper. On the other hand, there is at most one filtration $\Fil$ over each geometric point, because $\bar{\rho} _{\mf{m},w}$ is generic and totally indecomposable hence each $j$-th graded piece $gr _j = \Fil _j / \Fil _{j-1}$ has to be a lifting of $\bar{\chi} _{j,w}$ (see proofs of Lemma \ref{lem:tot-indec} and Lemma \ref{lem:tot-reduction}). This allows us to conclude.
\end{proof}

By Proposition above and Zariski Main Theorem we conclude that $f$ is finite and hence $\mc{G} = \Spec \mb{T}_{\mf{m},w} '$ for some $\mc{O}$-algebra $\mb{T}_{\mf{m},w}'$ finite over $\mb{T}_{\mf{m}}$.

\begin{coro}\label{coro:perfect}
The morphism $f: \Spec \mb{T}_{\mf{m},w} ' \ra \Spec \mb{T} _{\mf{m}}$ is a homeomorphism which induces an isomorphism of residual fields at each prime $\mf{p}' \in \Spec \mb{T}_{\mf{m},w} '$ with perfect residual field.
\end{coro}
\begin{proof}
It follows from the fact that geometric fibres of $f$ are of cardinal one.
\end{proof}
We define $\mb{T}_{\mf{m}}'$ to be the tensor product over $\mb{T}_{\mf{m}}$ of $\mb{T}_{\mf{m},w}'$ for all $w|p$. This is still an $\mc{O}$-algebra finite over $\mb{T}_{\mf{m}}$ with $\Spec \mb{T}_{\mf{m}}'$ homeomorphic to $\Spec \mb{T}_{\mf{m}}$.

\medskip

Consider the base-change of $\rho_\mf{m}$ to $\mb{T}_{\mf{m}}'$, that is $\rho_{\mf{m}} ' : G_F \ra \GL_n(\mb{T}_{\mf{m}}')$. By what we have said above, $\rho_{\mf{m},w}'$ can be conjugated to a triangular representation $\rho_{\mf{m},w}''$ for each $w|p$, which is also generic and totally indecomposable at each $w|p$, because $\bar{\rho} _{\mf{m},w}$ is by our assumption. By Corollary above, for each prime ideal $\mf{p}$ associated to an automorphic representation $\pi$ on $G(\mb{A})$, there exists a unique $\mf{p}'$ in $\mb{T}_{\mf{m}}'$ such that $\rho_{\mf{m}} ' / \mf{p}' \rho_{\mf{m}} ' [1/p]  \simeq \rho_{\mf{m}} / \mf{p} \rho_{\mf{m}}[1/p]$.

\medskip 

The above discussion leads us to the following definition
$$\Pi(\rho _{\mf{m},w})^{ord} := \Pi (\rho _{\mf{m},w} '')^{ord}$$
and similarly
$$\Pi(\rho _{\mf{m},w})_{I} := \Pi (\rho _{\mf{m},w} '')_I$$
for any $I$, in particular for $I= \emptyset$ which we shall use below. These are representations over $\mb{T}_{\mf{m}} '$. To conclude using our precedent results that the reduction modulo prime ideals of $\Pi(\rho _{\mf{m},w})^{ord}$ is well-behaved and compatible with the construction of Breuil-Herzig we need the following fact.

\begin{lemm}
For each prime ideal $\mf{p}$ of $\mb{T}_{\mf{m}}'$ which comes from a maximal ideal of $\mb{T}_{\mf{m}}'[1/p]$, the representation $\Pi (\rho _{\mf{m},w} '')^{ord} / \mf{p}[1/p]$ does not depend on the chosen triangulation $\rho _{\mf{m},w} ''$ of $\rho _{\mf{m},w} '$ (where by triangulation of $\rho _{\mf{m},w} '$ we mean a triangular representation which can be conjugated to $\rho _{\mf{m},w} '$).
\end{lemm}
\begin{proof}
By Proposition \ref{prop:comp} we deal with the construction of Breuil-Herzig and hence we can use facts from \cite{bh}. We have to prove that for any triangulation $\rho _{\mf{m},w} ''$ the reduction $\rho _{\mf{m},w} '' / \mf{p}$ is a good conjugate of $\rho _{\mf{m},w} / \mf{p}$ (Definition 3.2.4 in \cite{bh}). This would give our claim by Lemma 3.3.5 of \cite{bh}. By our assumption that $\bar{\rho} _{\mf{m},w}$ is generic triangular and totally indecomposable, any triangular lift $\rho$ of $\bar{\rho} _{\mf{m},w}$ is totally indecomposable and generic by Lemma \ref{lem:tot-reduction}. Then we conclude by (3) of Lemma \ref{lem:tot-indec} that each triangulation of $\rho _{\mf{m},w} '/\mf{p}$ (in particular $\rho _{\mf{m},w} '' / \mf{p}$) is a good conjugate of $\rho _{\mf{m},w} ' /\mf{p}$.

\end{proof}

We summarize our efforts so far in the following theorem.

\begin{theo}
Let $\mf{m}$ be an ordinary non-Eisenstein ideal of $\mb{T}$ such that $\bar{\rho} _{\mf{m}, w}$ is totally indecomposable and generic for each $w|p$ in $F$. Then we have for any prime ideal $\mf{p'}$ of $\mb{T}_{\mf{m}}'$ (with the inverse image $\mf{p}$ in $\mb{T}_{\mf{m}}$) which comes from a maximal ideal of $\mb{T}_{\mf{m}}'[1/p]$:
$$\Pi (\rho' _{\mf{m},w}) ^{ord} / \mf{p}' \Pi (\rho' _{\mf{m},w}) ^{ord} [1/p] \simeq \Pi (\rho _{\mf{m},w} / \mf{p} \rho _{\mf{m},w}[1/p]) ^{ord}$$
\end{theo}

Similar compatibilities with reduction modulo prime ideals hold for $\Pi(\rho _{\mf{m},w})_{I}$.

\subsection{On the pro-modular Fontaine-Mazur conjecture}

We come back to our general formalism which we will apply to $\Pi(\rho _{\mf{m},w})_{\emptyset}$. We assume that $\mf{m}$ is a non-Eisenstein ordinary ideal of $\mb{T}$ such that $\bar{\rho}_{\mf{m},w}$ is triangular, generic and totally indecomposable for each $w|p$. We take $\mb{T}_{\mf{m}}'$ to be $\mb{T}_{\mf{m}}'$ from preceding sections. We start with two lemmas:

\begin{lemm}
Let $\psi _1,\psi _2 : \Gal(\overline{\mb{Q}}_p / \mb{Q}_p) \ra E$ be two de Rham characters such that $\psi _1 \psi _2 ^{-1} \notin \{1,\varepsilon, \varepsilon ^{-1}\}$ and let $0 \ra \psi _1 \ra V \ra \psi _2 \ra 0$ be the non-split extension (there is a unique one; see below). Suppose that $V$ is de Rham. Then $\HT(\psi _1) < \HT(\psi _2)$ (normalizing $\HT(\varepsilon) = -1$).
\end{lemm}
\begin{proof}
The fact that there is a unique extension of $\psi _1$ by $\psi _2$ follows from the fact that $H^1 = H^1(\Gal(\overline{\mb{Q}}_p / \mb{Q}_p), \psi _1 \psi _2 ^{-1})$ is of dimension 1 because $\psi _1 \psi _2 ^{-1}$ is generic. Observe that $V \in H^1$. One can define the Selmer group $H^1 _g = H^1 _g(\Gal(\overline{\mb{Q}}_p / \mb{Q}_p), \psi _1 \psi _2 ^{-1})$ which measures whether $V$ is de Rham (we refer the reader to Chapter II of \cite{ber-t}; Definition is given before Proposition 2.17). By Corollary 2.18 of \cite{ber-t} we see that $V \in H^1 _g$. Hence $H^1 _g$ is of dimension one. But Proposition 2.19 of \cite{ber-t} gives us a formula for the dimension of $H^1 _g$, by which we infer in our case that $dim H^1 _g = 1$ is equal to the number of negative Hodge-Tate numbers (compare with the discussion after Proposition 2.19 in \cite{ber-t}). Hence $\HT(\psi _1) < \HT(\psi _2)$. 
\end{proof}

Recall that we have defined the character $\theta$ in Section 4.1. For $\GL_n$ this character is simply $diag(z_1,...,z_n) \mapsto \prod _i z_i ^{1-i}$.

\begin{lemm}\label{lem:dom}
Let $\rho : \Gal(\overline{\mb{Q}}_p/ \mb{Q}_p) \ra \GL_n(E)$ be a de Rham, triangular, totally indecomposable, generic Galois representation. Then the character $\chi _{\rho} \cdot (\varepsilon ^{-1} \circ \theta)$ is locally algebraic dominant.
\end{lemm}
\begin{proof}
Triangularity permits us to define $\chi _{\rho}$. It is clear that the character is locally algebraic because $\rho$ is de Rham. We conclude that $\chi _{\rho} \cdot (\varepsilon ^{-1} \circ \theta)$ is dominant by applying the lemma above to each pair of consecutive characters on the diagonal of $\rho$ (which we can do because $\rho$ is totally indecomposable and generic).
\end{proof}

We can now check that for representations $\Pi(\rho _{\mf{m},w})_{\emptyset}$ hypothesis (H1) holds:

\medskip

\noindent (H1): We have to check that if $\mf{p}$ is a prime ideal of $\mb{T}_{\mf{m}}'$ corresponding to the Galois representation $\rho_{\mf{p}}$ which is de Rham and regular at all places $w|p$, then locally algebraic vectors in $\Pi(\rho _{\mf{m},w})_{\emptyset} / \mf{p}[1/p] = \Pi(\rho _{\mf{p},w})_{\emptyset}[1/p]$ are non-zero. Indeed, the locally algebraic vectors in $\Pi(\rho_{\mf{p},w})_{\emptyset}[1/p]$ are non-zero because it is the representation induced from the locally algebraic dominant character $\chi = \chi _{\rho} \otimes (\varepsilon ^{-1} \circ \theta)$ (by Lemma \ref{lem:dom}); to see it we write $\chi = \chi _{sm} \delta _W$ for this character, where $\chi _{sm}$ is smooth and $\delta _W$ is algebraic corresponding to an irreducible algebraic representation $W$ of $G(\mb{Q}_p)$. We have $W = (\Ind _{B^-(\mb{Q}_p)} ^{G(\mb{Q}_p)} \delta _W)^{alg}$. Then the universal completion of the locally algebraic representation $(\Ind _{B^-(\mb{Q}_p)} ^{G(\mb{Q}_p)} \chi _{sm} )^{sm} \otimes W$ is equal to $(\Ind _{B^-(\mb{Q}_p)} ^{G(\mb{Q}_p)} \chi )^{\mc{C}^0} = \Pi(\rho _{\mf{p},w})_{\emptyset}[1/p]$ because $\chi$ is unitary (we inject the locally algebraic induction into the continuous induction by sending $f_ {sm} \otimes f_{alg}$ to $f_{sm}\cdot f_{alg}$, where $f_{sm}$,$f_{alg}$ are functions on smooth, respectively algebraic part). In particular, the set of locally algebraic vectors in $\Pi(\rho_{\mf{p},w})_{\emptyset}[1/p]$ is non-empty. The fact that $\Pi(\rho_{\mf{m},w})_{\emptyset}/\mf{m}'$ is of finite length is clear from the definition.




\medskip

For $\mf{p} \in P_{autom} ^{cris}$ as in Proposition \ref{prop:dense}, we know that each $\rho _{\mf{p},w}$ is crystalline triangularisable. Our hypothesis on $\bar{\rho}_{\mf{m},w}$ implies that $\rho_{\mf{p},w}$ is also totally indecomposable (Lemma \ref{lem:tot-reduction}) and generic. So we may unambigously associate to it a character $\chi_{\rho_{\mf{p},w}}$ of $T_n(\mb{Q}_p)$.

\medskip

Let us recall a classical local-global compatibility result.
\begin{lemm}\label{lem:automorphic}
Fix $\mf{p} \in P_{autom} ^{cris}$. Let $W$ be the irreducible algebraic representation of $G(\mb{Q}_p)$ such that $H^0(K^p, \mc{V}_W)[\mf{p}] \not = 0$. Let $\pi$ be an automorphic representation such that $\pi _f ^{K^p} \subset H^0(K^p, \mc{V}_W)[\mf{p}]$. Then 
  
$$W ^{\vee} = \otimes _{w|p} (\Ind _{B^-(\mb{Q}_p)} ^{\GL_n(\mb{Q}_p))} (\chi _{\rho _{\pi},w} \cdot (\varepsilon ^{-1} \circ \theta))_{alg})^{alg}$$
$$\pi _p = \otimes _{w|p} (\Ind _{B^-(\mb{Q}_p)} ^{\GL_n(\mb{Q}_p)} (\chi _{\rho _{\pi},w} \cdot (\varepsilon ^{-1} \circ \theta))_{sm})^{sm}$$
where we have denoted by $(.)_{sm}$ (respectively, $(.)_{alg}$) the smooth (resp. algebraic) part of the character.
\end{lemm}
\begin{proof}

The first isomorphism follows from Corollary 2.7.8(i) of \cite{ge} with the following dictionary:
\begin{itemize}
	\item our $W$ is Geraghty's $M_{\lambda}$, therefore $W^{\vee} = \Ind _{B^-} ^{\GL_n} (w_0 \lambda) ^{-1}$
	\item his $\lambda = (\lambda _{\tau})_{\tau : F^+ \hookrightarrow E} = (\lambda _w) _{w|p}$ since $p$ was assumed to be totally split in $F$.
	\item for each $w|p$, loc. cit. tells us that $(\chi _{\rho _{\pi, w}})_{alg} = (w_0 \lambda _w)^{-1} \cdot \theta$
\end{itemize}

The second isomorphism follows from Corollary 2.7.8(ii) of \cite{ge} and the first formula on p. 27 of \cite{ge} (proof of Lemma 2.7.5). Namely, loc. cit. tells us that $\pi _w$ is the unramified subquotient of $(n-\Ind _{B(\mb{Q}_p)} ^{\GL_n(\mb{Q}_p)} (\chi _{\rho _\pi})_{sm}) \otimes |\det| ^{(n-1)/2}$ (normalized induction). But the genericity of $\rho _{\mf{p},w}$ implies that
$$\pi _w = (n-\Ind _{B(\mb{Q}_p)} ^{\GL_n(\mb{Q}_p)} (\chi _{\rho _\pi})_{sm}) \otimes |\det| ^{(n-1)/2}$$
and smooth representation theory tells us that this is also 
$$(n-\Ind _{B^-(\mb{Q}_p)} ^{\GL_n(\mb{Q}_p)} (\chi _{\rho _\pi})_{sm} )\otimes |\det| ^{(n-1)/2} =  (\Ind _{B^-(\mb{Q}_p)} ^{\GL_n(\mb{Q}_p)} (\chi _{\rho _\pi})_{sm} \delta _B ^{-1/2}) \otimes |\det| ^{(n-1)/2} $$
where $\delta _B$ is the modulus character. We conclude by observing that 
$$ \delta _B ^{-1/2} \cdot |\det| ^{(n-1)/2} = (\varepsilon ^{-1} \circ \theta) _{sm} : (z_1,...,z_n) \mapsto \prod |z_i| ^{i-1}$$

\end{proof}

Using $\Pi(\rho _{\mf{m},w})_{\emptyset}$ we can make use of our formalism (Theorem \ref{thm:FM}) to get the pro-modular Fontaine-Mazur conjecture in the following form.

\begin{theo}
Let $\mf{m}$ be an ordinary non-Eisentein ideal of $\mb{T}$ such that $\bar{\rho} _{\mf{m},w}$ is totally indecomposable and generic for all $w|p$. Let $\rho$ be pro-modular with respect to $\mb{T}_{\mf{m}}$ and de Rham regular. Then $\rho$ is modular.
\end{theo}
\begin{proof}
To conclude by Theorem \ref{thm:FM} we have to check that hypotheses (H2) and (H3) hold for $\Pi(\rho _{\mf{m},w})_{\emptyset}$. The hypothesis (H2) says that there exists an allowable set for $\Pi(\rho _{\mf{m},w})_{\emptyset}$, which means that there exists a dense set of prime ideals $\mf{p}$ in the Hecke algebra $\mb{T}_{\mf{m}}'$ for which the associated Galois representation $\rho _{\mf{p},w}$ ($w|p$ is a split place) gives a Banach representation $\Pi(\rho _{\mf{p},w})_{\emptyset}$ with an injection
$$\otimes _{w|p} \Pi(\rho _{\mf{p},w})_{\emptyset} \hookrightarrow \widehat{H}^0(K^p)_E$$
We can prove it for the set $P_{autom} ^{cris}$ which is dense by Proposition \ref{prop:dense}. Indeed, if $\mf{p}$ corresponds to a classical automorphic representation $\pi$ with the Galois representation $\rho _{\pi}$. Following Lemma \ref{lem:automorphic} we put $\chi _w = \chi _{\rho _{\pi},w} \cdot (\varepsilon ^{-1} \circ \theta)$. Take $\chi = \otimes _{w|p} \chi _w$ and write $\chi = \chi _{sm} \delta _W$ as above in the verification of (H1). Then by the description of locally algebraic vectors of $\widehat{H}^0(K^p)_{E, l.alg}$ from Proposition \ref{prop:alg-vectors} we see that $W^{\vee} \otimes (\Ind _{B^-(\mb{Q}_p)} ^{G(\mb{Q}_p)} \chi _{sm}) ^{sm}$ injects into $\widehat{H}^0(K^p)_{E, l.alg}$, (we use here Lemma \ref{lem:automorphic}). Hence taking the completion we see that the universal completion $\widehat{W^{\vee} \otimes (\Ind _{B^-(\mb{Q}_p)} ^{G(\mb{Q}_p)} \chi _{sm}) ^{sm}}$ of $W^{\vee} \otimes (\Ind _{B^-(\mb{Q}_p)} ^{G(\mb{Q}_p)} \chi _{sm}) ^{sm}$ sits in $\widehat{H}^0(K^p)_{E}$. But because $\chi$ is unitary, we have
$$\widehat{W^{\vee} \otimes (\Ind _{B^-(\mb{Q}_p)} ^{G(\mb{Q}_p)} \chi _{sm}) ^{sm}} = (\Ind _{B^-(\mb{Q}_p)} ^{G(\mb{Q}_p)} \chi) ^{\mc{C}^0} = \otimes _{w|p} \Pi(\rho _{\mf{p},w})_{\emptyset}$$
by which we conclude.

\medskip

The hypothesis (H3)[$\mf{p}'$] ($\mf{p}'$ is the prime ideal of $\mb{T}_{\mf{m}}'$ corresponding to $\rho$ by our pro-modularity assumption and Corollary \ref{coro:perfect}) says that we have a closed injection
$$\Pi(\rho _{\mf{p}',w})_{\emptyset} \hookrightarrow \widehat{H}^{0}(K^p)_E [\mf{p}']$$
This is clear in our context because $\rho _{\mf{p}',w}$ is generic and hence $\Pi(\rho _{\mf{p}',w})_{\emptyset}$ is irreducible (see Theorem 3.1.1(ii) in \cite{bh}).

\medskip

This allows us to conclude.

\end{proof}

We can get a more explicit result by using eigenvarieties. In \cite{em1} Emerton has constructed the eigenvariety $\mc{X}$ associated to the group $G$ using completed cohomology. We do not recall here this construction explicitely, but let us mention that $\mc{X}$ parametrises (certain) pro-modular representations. Let $\mc{X}[\bar{\rho}_{\mf{m}}]$ be the $\bar{\rho}_{\mf{m}}$-part of the eigenvariety associated to $U(n)$. In particular every point $x \in \mc{X}[\bar{\rho}_{\mf{m}}]$ is pro-modular with respect to $\mb{T}_{\mf{m}}$. We denote by $\lambda _x$ its corresponding Hecke character and by $\rho_x$ its associated Galois representation.

The above theorem implies the following result

\begin{coro}
Let $\mf{m}$ be an ordinary non-Eisenstein ideal of $\mb{T}_{\mf{m}}$ such that $\bar{\rho} _{\mf{m},w}$ is totally indecomposable and generic for all $w|p$. Let $x$ be an $E$-point on the eigenvariety $\mc{X}[\bar{\rho}_{\mf{m}}](E)$ such that for each place $w|p$ the representation $\rho _{x,w}$ is regular and de Rham. Then $x$ is classical. 
\end{coro}
\begin{proof}
The conclusion that $x$ is modular follows from the theorem above. To see that it is classical, it is enough to observe that because our Galois represetation is generic at places dividing $p$, every refinement is accessible. We do not explain these notions here. 
\end{proof}

We obtain a similar result (i.e. pro-modularity with additional assumptions implies modularity) for $U(2)$ in \cite{cs} in the setting of irreducible Galois representations (rather than triangular).


\end{document}